\documentclass[12pt, reqno]{amsart}
\usepackage{amsmath, amsthm, amscd, amsfonts, amssymb, graphicx, color}
\usepackage[bookmarksnumbered, colorlinks, plainpages]{hyperref}
\input{mathrsfs.sty}

\newtheorem{theorem}{Theorem}[section]
\newtheorem{lemma}[theorem]{Lemma}

\newtheorem{corollary}[theorem]{Corollary}
\theoremstyle{definition}

\theoremstyle{remark}
\newtheorem{remark}[theorem]{Remark}
\numberwithin{equation}{section}

\begin{document}

\title[Equality in a generalized Dunkl--Williams inequality]{Characterization of equality in a generalized Dunkl--Williams inequality}

\author[M.S. Moslehian, F. Dadipour]{Mohammad Sal Moslehian and Farzad Dadipour}

\address{Department of Pure Mathematics, Center of Excellence in
Analysis on Algebraic Structures (CEAAS), Ferdowsi University of
Mashhad, P. O. Box 1159, Mashhad 91775, Iran.}
\email{dadipoor@yahoo.com} \email{moslehian@ferdowsi.um.ac.ir and
moslehian@ams.org}
\urladdr{\url{http://profsite.um.ac.ir/~moslehian/}}



\begin{abstract}
We establish a generalization of the Dunkl--Williams inequality and its inverse in the framework of
Hilbert $C^*$-modules and characterize the equality case. As applications, we get some new results and some known results
due to Pe\v{c}ari\'c and Raji\'c [Linear Algebra Appl. 425 (2007), no. 1, 16--25].
\end{abstract} \maketitle


\section{Introduction and Preliminaries}

There are many interesting refinements and reverses of the
triangle inequality in normed linear spaces. In this direction some
authors improved the well-known Dunkl--Williams inequality
\cite{D-W} that states:
\begin{eqnarray}\label{D-W}
\left\|\frac{x}{\|x\|}-\frac{y}{\|y\|}\right\|\leq\frac{4\|x-y\|}{\|x\|+\|y\|}
\end{eqnarray}
for any two non-zero elements $x,y$ in a normed linear space.
Maligranda \cite{Malig} presented the following refinement of
\eqref{D-W} (see also \cite{MAL})
\begin{eqnarray}\label{Malig}
\left\|\frac{x}{\|x\|}-\frac{y}{\|y\|}\right\|
\leq\frac{\|x-y\|+|\,\|x\|-\|y\|\,|}{\max \{ \|x\|,\|y\| \} }\,.
\end{eqnarray}
A reverse of inequality \eqref{Malig} was given by Mercer
\cite{Merc} as follows
\begin{eqnarray}\label{Merc}
\left\|\frac{x}{\|x\|}-\frac{y}{\|y\|}\right\|
\geq\frac{\|x-y\|-|\,\|x\|-\|y\|\,|}{\min \{\|x\|,\|y\|\} }\,.
\end{eqnarray}
In \cite{Kato} Kato et al. improved  the triangle inequality and
provided a reverse by showing that
\begin{eqnarray}\label{Kato1}
\left\|\sum_{j=1}^{n}
x_j\right\|+\left(n-\left\|\sum_{j=1}^{n}\frac{x_j}{\|x_j\|}\right\|\
\right)\min_{1\leq i\leq
n}\|x_i\|\leq\sum_{j=1}^{n}\|x_j\|\\
\left\|\sum_{j=1}^{n}x_j\right\|+\left(n-\left\|\sum_{j=1}^{n}\frac{x_j}{\|x_j\|}\right\|\
\right)\max_{1\leq i\leq n}\|x_i\|\geq\sum_{j=1}^{n}\|x_j\|
\label{Kato2}
\end{eqnarray}
for all non-zero elements $x_1,\cdots,x_n$ of a normed linear space.

In \cite{D-F-M} the authors presented several operator versions of the Dunkl--Williams inequality
with respect to the $p$-angular distance for operators. Pe\v{c}ari\'c and Raji\'c \cite{Pec-Raj1} sharped inequalities
\eqref{Kato1} and \eqref{Kato2} (when $n>2$) and they generalized
inequalities \eqref{Malig} and \eqref{Merc} by showing that
\begin{eqnarray}\label{Pec-Raj1}
\left\|\sum_{j=1}^{n}\frac{x_j}{\|x_j\|}\right\|\leq\min_{1\leq
i\leq n}\left\{\frac{1}{\|x_i\|}\left(\
\left\|\sum_{j=1}^{n}x_j\right\|+\sum_{j=1}^{n}|\,\|x_j\|-\|x_i\|\,|\right)\right\}\\
\left\|\sum_{j=1}^{n}\frac{x_j}{\|x_j\|}\right\|\geq\max_{1\leq
i\leq n}\left\{\frac{1}{\|x_i\|}\left(\
\left\|\sum_{j=1}^{n}x_j\right\|-\sum_{j=1}^{n}|\,\|x_j\|-\|x_i\|\,|\right)\right\}\label{Pec-Raj2}
\end{eqnarray}
for all non-zero elements $x_1,\cdots,x_n$ of a normed linear space.

In \cite{Drag} Dragomir obtained a generalization of inequalities
\eqref{Pec-Raj1} and \eqref{Pec-Raj2} by replacing arbitrary scalars
instead of $\frac{1}{\|x_i\|}$\,\,$(i=1,\cdots,n)$. The equality case in
the triangle inequality and its refinements for Banach space
operators was studied by many authors. Barra and Boumzgour
\cite{Bar-Boum} presented a characterization of triangle equality
for Hilbert space operators. Pe\v{c}ari\'c and Raji\'c
\cite{Pec-Raj2} presented some equivalent conditions for the case of
equality in inequalities \eqref{Pec-Raj1} and \eqref{Pec-Raj2} for
elements of pre-Hilbert $C^*$-modules.

In this paper we establish a generalization of the Dunkl--Williams inequality and its inverse in the framework of
Hilbert $C^*$-modules and characterize the equality case. As applications, we get some new and known results
in the case of equality.

The notion of pre-Hilbert $C^*$-module is a generalization of that
of Hilbert space in which the field of scalars $\mathbb{C}$ is
replaced by a $C^*$-algebra. The formal definition is as follows.

A complex linear space $\mathscr{X}$ that is a right module over a
$C^*$-algebra $\mathscr{A}$
 is called a pre-Hilbert $\mathscr{A}$-module if there is an $\mathscr{A}$-valued inner product on $\mathscr{X}$,
 i.e. a map $\langle . , . \rangle :\mathscr{X}\times \mathscr{X}\longrightarrow\mathscr{A}$ with the following properties:\\
 (i) $\langle x,x \rangle\geq0$, for $x\in \mathscr{X}$\\
(ii) $\langle x,x \rangle=0$ if and only if $x=0$\\
(iii) $\langle x,\alpha y+\beta z \rangle=\alpha\langle x,y \rangle+\beta\langle x,z \rangle$, for $x,y,z\in \mathscr{X}$, $\alpha,\beta\in \mathbb{C}$\\
(iv) $\langle x,ya \rangle=\langle x,y \rangle a$, for $x,y\in \mathscr{X}$, $a\in \mathscr{A}$\\
(v) $\langle x,y \rangle=\langle y,x \rangle^*$, for $x,y\in \mathscr{X}$\\
One defines a norm on $\mathscr{X}$ by $\|x\|=\|\langle x,x
\rangle\|^\frac{1}{2}$, $x\in \mathscr{X}$. A pre-Hilbert $\mathscr{A}$ which
is complete with respect to its norm is called a Hilbert
$C^*$-module. Clearly every inner product space is a pre-Hilbert $\mathbb{C}$-module. Also every $C^*$-algebra $\mathscr{A}$ is a Hilbert
$\mathscr{A}$-module under the inner product given by $\langle a,b
\rangle=a^*b$. The Banach space $B(H_1,H_2)$ of all bounded linear
operator between Hilbert spaces $H_1$ and $H_2$ is a Hilbert
$B(H_1)$-module under the inner product $\langle T,S \rangle=T^*S$.

Throughout this paper $\mathscr{A}$ denotes a unital $C^*$-algebra with the unit $e$. We refer the reader to \cite{Mur} for undefined notions on $C^*$-algebra theory and to \cite{Lan} for more information on
Hilbert $C^*$-modules.


\section{Main results}

We start this section with the following useful Lemma due to
Aramba\v{s}i\'c and Raji\'c \cite{Aram-Raj}, which characterizes the
generalized triangle equality for finitely many elements of a
pre-Hilbert $C^*$-module.

\begin{lemma}\label{Aram-Rajic}
Let $\mathscr{X}$ be a pre-Hilbert $\mathscr{A}$-module and
$x_1,\cdots,x_n$ non-zero elements of $\mathscr{X}$. Then the
equality $\|x_1+\cdots+x_n\|=\|x_1\|+\cdots+\|x_n\|$ holds if and
only if there is a state $\varphi$ on $\mathscr{A}$ such that
$\varphi \langle x_i,x_n \rangle=\|x_i\| \|x_n\|$ for all $1\leq
i\leq n-1$.
\end{lemma}

The next lemma is interesting on its own right. Recall that an
element $a\in \mathscr{A}$ is called coisometry if $aa^*=e$.

\begin{lemma}\label{lemma}
 Let $\mathscr{X}$ be a pre-Hilbert $\mathscr{A}$-module, $x\in \mathscr{X}$ and $a\in \mathscr{A}$ be
 a scalar multiple of a coisometry. Then\\
{\rm (i)} $\|xa\|=\|x\|\ \|a\|$;\\
{\rm (ii)}$x=0$ or $a=0$ if $xa=0$.
\end{lemma}

\begin{proof}
(i) Let $a=\lambda u$, for some scalar $\lambda\in\mathbb{C}$ and some
coisometry $u\in\mathscr{A}$. First we note that
$\|a\|=\|a^*\|=\|aa^*\|^{\frac{1}{2}}=\|\lambda\bar{\lambda}uu^*\|^{\frac{1}{2}}=|\lambda|$.
We have
\begin{eqnarray*}
\|xa\|&=&\|\langle xa,xa \rangle\|^{\frac{1}{2}}=\|a^*\langle x,x
\rangle a\|^{\frac{1}{2}}=\|\langle x,x \rangle^{\frac{1}{2}}a\|
=\|a^*\langle x,x \rangle^{\frac{1}{2}}\|\\
&=&\|\langle x,x \rangle^{\frac{1}{2}}aa^*\langle x,x
\rangle^{\frac{1}{2}}\|^{\frac{1}{2}}=|\lambda|\ \|\langle x,x
\rangle\|^{\frac{1}{2}}=\|x\| \|a\|\,.
\end{eqnarray*}
(ii) It follows from part (i).
\end{proof}

Now we can establish a generalization of the Dunkl--Williams
inequality and its reverse in a pre-Hilbert $C^*$-module.
Our results generalize inequalities
\eqref{Pec-Raj1}, \eqref{Pec-Raj2} and some results due to Dragomir
\cite{Drag} for elements of pre-Hilbert $C^*$-modules.

\begin{theorem}\label{T1}
Let $\mathscr{X}$ be a pre-Hilbert $\mathscr{A}$-module. If $x_j\in \mathscr{X}$ and
$a_j\in\mathscr{A}$ for $j=1,\cdots,n$ such that $a_j, a_j-a_i$ are
scalar multiples of coisometries, then
\begin{eqnarray}\label{L1}
{\rm (i)}\ \left\|\sum_{j=1}^{n} x_ja_j\right\|\leq \min_{1\leq
i\leq n}\left\{\left\|\sum_{j=1}^{n} x_j\right\|
\|a_i\|+\sum_{j=1}^{n}\|x_j\|\
\|a_j-a_i\|\right\}\,,\\
{\rm (ii)}\ \left\|\sum_{j=1}^{n} x_ja_j\right\|\geq \max_{1\leq
i\leq n}\left\{\left\|\sum_{j=1}^{n} x_j\right\|
\|a_i\|-\sum_{j=1}^{n}\|x_j\|\ \|a_j-a_i\|\right\}.\label{L2}
\end{eqnarray}
\end{theorem}

\begin{proof}
(i) For any fixed $i\,\,(1\leq i\leq n)$ we have
$$\sum_{j=1}^{n}x_ja_j=\sum_{j=1}^{n}x_ja_i+\sum_{j=1}^{n}x_j(a_j-a_i).$$
Also from Lemma \ref{lemma} we get $\|\sum_{j=1}^{n} x_j
a_i\|=\|\sum_{j=1}^{n} x_j\|\ \|a_i\|$ and
$\|x_j(a_j-a_i)\|=\|x_j\|\ \|a_j-a_i\|$. Hence
\begin{eqnarray*}
\left\|\sum_{j=1}^{n}
x_ja_j\right\|&=&\left\|\sum_{j=1}^{n}x_ja_i+\sum_{j=1}^{n}x_j(a_j-a_i)\right\|\leq\left\|\sum_{j=1}^{n}
x_ja_i\right\|+\sum_{j=1}^{n}\|x_j(a_j-a_i)\|\\
&=&\left\|\sum_{j=1}^{n} x_j\right\| \|a_i\|+\sum_{j=1}^{n}\|x_j\|\
\|a_j-a_i\|\,.
\end{eqnarray*}
Taking the minimum over $i=1,\cdots,n$ we deduce inequality
\eqref{L1}.\\
(ii) Fix $i\,\,(i=1,\cdots,n)$. From the triangle inequality and Lemma
\ref{lemma} we have
\begin{eqnarray*}
\left\|\sum_{j=1}^{n}
x_ja_j\right\|&=&\left\|\sum_{j=1}^{n}x_ja_i-\sum_{j=1}^{n}x_j(a_i-a_j)\right\|\geq\left\|\sum_{j=1}^{n}x_ja_i\right\|-\left\|\sum_{j=1}^{n}x_j(a_i-a_j)\right\|\\
&\geq&\left\|\sum_{j=1}^{n}x_ja_i\right\|-\sum_{j=1}^{n}\left\|x_j(a_i-a_j)\right\|\\
&=&\left\|\sum_{j=1}^{n}x_j\right\|\|a_i\|-\sum_{j=1}^{n}\|x_j\|\
\|a_i-a_j\|\,.
\end{eqnarray*}
From this we obtain inequality \eqref{L2} by taking the maximum over
$i=1,\cdots,n$.
\end{proof}

Following two results provide some equivalent conditions for the
equality case in Theorem \ref{T1}.

\begin{lemma}\label{T2}
Let $\mathscr{X}$ be a pre-Hilbert $\mathscr{A}$-module, $x_1,\cdots,x_n$ be non-zero
elements of $\mathscr{X}$ satisfying $\sum_{j=1}^{n} x_j\neq0$ and $a_1,\cdots,a_n$ be non-zero elements of $\mathscr{A}$
such that $a_i\neq a_j$ for some $i,j$
and the elements $a_j,a_j-a_i$ are scalar multiples of coisometries for all
$i,j$. Then for any $i\,\,(1\leq i\leq n)$ the following two
statements are equivalent:
\begin{eqnarray}\label{L3}
{\rm (i)}\ \left\|\sum_{j=1}^{n}
x_ja_j\right\|=\left\|\sum_{j=1}^{n} x_j\right\|
\|a_i\|+\sum_{j=1}^{n}\|x_j\|\ \|a_j-a_i\|\,.\ \ \ \ \ \ \ \ \ \ \ \ \
\ \ \ \ \ \ \ \ \ \ \ \ \ \ \ \ \ \
\end{eqnarray}
{\rm (ii)}\ There is a state $\varphi$ on $\mathscr{A}$ such that
$$\sum_{j=1}^{n} \varphi(a_i^*\langle x_j,x_k \rangle(a_k-a_i))=\left\|\sum_{j=1}^{n} x_j\right\|\|a_i\|\|x_k\|\|a_k-a_i\|$$
for all $k\,\,(k=1,\cdots,n)$ satisfying $a_k\neq a_i$.
\end{lemma}

\begin{proof}
Let us fix $i\,\,(1\leq i\leq n)$. One can observe that \eqref{L3} is
equivalent to
\begin{eqnarray}\label{L4}
\left\|\sum_{j=1}^{n}x_ja_i+\sum_{j=1}^{n}x_j(a_j-a_i)\right\|=\left\|\sum_{j=1}^{n}
x_j\right\|\|a_i\|+\sum_{j=1}^{n}\|x_j\|\|a_j-a_i\|\,.
\end{eqnarray}
From the assumption of the theorem there exists a nonempty maximal subset
$\{j_1, \cdots, j_m \}$ of $\{ 1,\cdots,n\}$ for some $1\leq m\leq n$
such that $a_{j_k}\neq a_i$ for all $1\leq k\leq m$. Hence
\eqref{L4} holds if and only if
\begin{eqnarray}\label{L5}
\left\|\sum_{j=1}^{n}x_ja_i+\sum_{k=1}^{m}x_{j_k}(a_{j_k}-a_i)\right\|=\left\|\sum_{j=1}^{n}
x_j\right\|\|a_i\|+\sum_{k=1}^{m}\|x_{j_k}\|\ \|a_{j_k}-a_i\|\,.
\end{eqnarray}
From Lemma \ref{lemma} we have $\sum_{j=1}^{n}x_ja_i\neq0 ,
 x_{j_k}(a_{j_k}-a_i)\neq0 , \left\|\sum_{j=1}^{n}x_ja_i\right\|=\left\|\sum_{j=1}^{n}
x_j\right\|\|a_i\|$ and $\|x_{j_k}(a_{j_k}-a_i)\|=\|x_{j_k}\|
\|a_{j_k}-a_i\|, (1\leq k\leq m)$.\\ Applying Lemma
\ref{Aram-Rajic} on the elements $x_{j_k}(a_{j_k}-a_i)\,\,(1\leq k\leq m)$ and $\sum_{j=1}^{n}x_ja_i$ we conclude that \eqref{L5} holds
if and only if there is a state $\varphi$ on $\mathscr{A}$ such that
$$\varphi\langle \sum_{j=1}^{n}x_ja_i,x_{j_k}(a_{j_k}-a_i)
\rangle=\left\|\sum_{j=1}^{n} x_j\right\|\|a_i\|\ \|x_{j_k}\|\
\|a_{j_k}-a_i\|\,\,\,(1\leq k\leq m)\,.$$
Hence
$$\sum_{j=1}^{n} \varphi(a_i^*\langle x_j,x_{j_k}
\rangle(a_{j_k}-a_i))=\left\|\sum_{j=1}^{n} x_j\right\|\|a_i\|\
\|x_{j_k}\|\ \|a_{j_k}-a_i\|\,\,\, (1\leq k\leq m)\,.$$

\end{proof}

\begin{lemma}\label{T3}
Let $\mathscr{X}$ be a pre-Hilbert $\mathscr{A}$-module, $x_1,\cdots,x_n$ be non-zero
elements of $\mathscr{X}$ satisfying $\sum_{j=1}^{n} x_j=0$ and the elements $a_1,\cdots,a_n$ be non-zero elements of $\mathscr{A}$
such that $a_i\neq a_j$ for some $i,j$
and $a_j-a_i$ are scalar multiples of coisometries for all
$i,j$. Then for any $i\,\,(1\leq i\leq n)$ the following two
statements are equivalent:

\begin{eqnarray}\label{L6}
{\rm (i)}\ \left\|\sum_{j=1}^{n}
x_ja_j\right\|=\sum_{j=1}^{n}\|x_j\|\|a_j-a_i\|\,.\ \ \ \ \ \ \ \ \ \ \
\ \ \ \ \ \ \ \ \ \ \ \ \ \ \ \ \ \ \ \ \ \ \ \ \ \ \ \ \ \ \ \ \ \
\ \ \ \ \ \
\end{eqnarray}
{\rm (ii)}\ There exist $1\leq l\leq n$ such that $a_l\neq a_i$ and
a state $\varphi$ on $\mathscr{A}$ such that
$$\varphi((a_l^*-a_i^*)\langle x_l,x_k \rangle(a_k-a_i))=\|a_l-a_i\|\|a_k-a_i\|\|x_l\|\|x_k\|$$
for all $k\,\,(k=1,\cdots,n)$ satisfying $k\neq l$ and $a_k\neq a_i$.
\end{lemma}

\begin{proof}
Let $i\,\,(1\leq i\leq n)$ be fixed. It follows from $\sum_{j=1}^{n} x_j=0$ that
\begin{eqnarray}\label{L7}
\left\|\sum_{j=1}^{n} x_j(a_j-a_i)\right\|=\sum_{j=1}^{n}\|x_j\|\
\|a_j-a_i\|
\end{eqnarray}
is equivalent to \eqref{L6}. Also let $\{j_1,\cdots,j_m\}$ be as in
the proof of Lemma \ref{T2}. We infer that \eqref{L7} holds if and
only if
\begin{eqnarray}\label{L8}
\left\|\sum_{k=1}^{m}
x_{j_k}(a_{j_k}-a_i)\right\|=\sum_{k=1}^{m}\|x_{j_k}\|\|a_{j_k}-a_i\|\,.
\end{eqnarray}
Using Lemma \ref{lemma} we have $x_{j_k}(a_{j_k}-a_i)\neq0$ and
$\|x_{j_k}(a_{j_k}-a_i)\|=\|x_{j_k}\|\ \|a_{j_k}-a_i\|\,\,
(k=1,\cdots,m)$. Applying Lemma \ref{Aram-Rajic} on the elements
$x_{j_k}(a_{j_k}-a_i)\,\,(k=1,\cdots,m)$ we deduce that \eqref{L8}
holds if and only if  there exists a state $\varphi$ on
$\mathscr{A}$ such that
$$\varphi\langle x_{j_l}(a_{j_l}-a_i),x_{j_k}(a_{j_k}-a_i) \rangle=\|a_{j_l}-a_i\|\ \|a_{j_k}-a_i\|\|x_{j_l}\|\|x_{j_k}\|$$
for some $1\leq l\leq m$ and for all $k\in\{1,\cdots,m\}\setminus\{l\}$.
Hence
$$\varphi((a_{j_l}^*-a_i^*)\langle x_{j_l},x_{j_k} \rangle(a_{j_k}-a_i))=\|a_{j_l}-a_i\|\ \|a_{j_k}-a_i\| \|x_{j_l}\|  \|x_{j_k}\|$$
for some $1\leq l\leq m$ and for all $k\in\{1,\cdots,m\}\setminus\{l\}$.
\end{proof}

Now we are ready to state the following theorem as an immediate consequence of Lemmas \ref{T2} and \ref{T3}. It characterizes
the generalized Dunkl--Williams equality in pre-Hilbert $C^*$-modules.

\begin{theorem}\label{T4}
Let $\mathscr{X}$ be a pre-Hilbert $\mathscr{A}$-module, $x_1,\cdots,x_n$ be non-zero
elements of $\mathscr{X}$ and $a_1,\cdots,a_n$ be non-zero elements of $\mathscr{A}$
such that $a_i\neq a_j$ for some $i,j$
and the elements $a_j, a_j-a_i$ are scalar multiples of coisometries for all
$i,j$. \\
{\rm (i)} If $\sum_{j=1}^{n} x_j\neq0$, then
$$\left\|\sum_{j=1}^{n} x_ja_j\right\|=\min_{1\leq k\leq
n}\left\{\left\|\sum_{j=1}^{n} x_j\right\|
\|a_k\|+\sum_{j=1}^{n}\|x_j\|\ \|a_j-a_k\|\right\}$$
if and only if there are $1\leq i\leq n$ and a state $\varphi$ on $\mathscr{A}$ such that\\
$$\sum_{j=1}^{n} \varphi(a_i^*\langle x_j,x_k
\rangle(a_k-a_i))=\|\sum_{j=1}^{n} x_j\|\|a_i\|\ \|x_k\|\|a_k-a_i\|$$
for all $k=1,\cdots,n$ satisfying $a_k\neq a_i$.\\
{\rm (ii)} If $\sum_{j=1}^{n} x_j=0$, then\\
$\left\|\sum_{j=1}^{n} x_ja_j\right\|=\min_{1\leq k\leq
n}\left\{\sum_{j=1}^{n}\|x_j\|\ \|a_j-a_k\|\right\}$ if and only if there are $i,l\in\{1,\cdots,n\}$ satisfying $a_i\neq a_l$ and a state $\varphi$ on $\mathscr{A}$ such that
$$\varphi((a_l^*-a_i^*)\langle x_l,x_k
\rangle(a_k-a_i))=\|a_l-a_i\|\|a_k-a_i\|\|x_l\|\|x_k\|$$ for all
$k=1,\cdots,n$ satisfying $k\neq l$ and $a_k\neq a_i$.
\end{theorem}

\begin{remark}
The condition that all elements $a_j$ and $a_j-a_i\,\,(i\neq j)$ are scalar multiples of coisometries is not restrictive. In fact, there are non-trivial concrete examples of elements $a_1, \ldots,a_n$ of some $C^*$-algebras satisfying this condition. A non-trivial example of a set of two elements is given in $M_2(\mathbb{C})$ by
$$a_1=\left[\begin{array}{cc}\alpha&0\\0&\beta\end{array}\right], \quad a_2=\left[\begin{array}{cc}\beta&0\\0&\alpha\end{array}\right]\,,$$
where $\alpha$ and $\beta$ are any complex numbers such that $|\alpha|=|\beta|$ and $\alpha^2\neq\beta^2$.\\
Now assume that $\mathscr{A}$ is a unital $C^*$-algebra, with the unit $e$, which has a halving projection $p$, i.e. a projection $p$ satisfying $p\sim e$ and $e-p\sim e$. Recall that two projections $p$ and $q$ are called (Murray-von Neumann) equivalent, denoted $p\sim q$, if there exists a partial isometry $v \in \mathscr{A}$ such that $p=v^*v$ and $q=vv^*$. A known example of a halving projection is $p(x_1, x_2, x_3, x_4, \ldots)=(0, x_2, 0, x_4, 0, \ldots)$ in $\mathbb{B}(\ell_2)$. Halving projections are useful due to allow one to consider some matrix structures inside the underlying $C^*$-algebra, see e.g. \cite[Chapter 5]{WEG} and \cite[Theorem 4.2]{MOS}.\\
Now assume that $\mathscr{A}$ has a halving projection, then it has $n$ mutually orthogonal halving projections $p_1, p_2, \ldots, p_n$ (see \cite[Lemma 5.3.5]{WEG}). Hence there are partial isometries $v_j\,\,(1\leq j\leq n)$ such that $p_j=v_j^*v_j$ and $e=v_jv_j^*$. It follows from $p_jp_k=0$ that $v_j^*v_jv_k^*v_k=0$, whence $v_jv_k^*=0$ for all $1\leq j \neq k \leq n$. Hence $(v_j-v_k)(v_j-v_k)^*=v_jv_j^*-v_jv_k^*-v_kv_j^*+v_kv_k^*=2e$. Thus $v_1, v_2, \ldots, v_n$ can be considered as the required elements.\\
It follows from \cite[Proposition 3.2.4]{SAK} that if $\mathscr{A}$ is a properly infinite $W^*$-algebra, then there exists a sequence $\{p_n\}$ of mutually orthogonal projections in $\mathscr{A}$ with $p_n \sim e$. This therefore provides an infinite sequence of requested elements as above.
\end{remark}

The next result characterizes the equality case in an inequality due
to Dragomir \cite{Drag} in Pre-Hilbert $C^*$-modules.

\begin{corollary}\label{C1}
Let $\mathscr{X}$ be a pre-Hilbert $\mathscr{A}$-module,
$x_1,\cdots,x_n$ be non-zero elements of $\mathscr{X}$ and
$\alpha_1,\cdots,\alpha_n$ be non-zero scalars
satisfying $\alpha_i\neq\alpha_j$ for some i,j.\\
{\rm (i)} If $\sum_{j=1}^{n} x_j\neq0$, then
$$\left\|\sum_{j=1}^{n}\alpha_jx_j\right\|=\min_{1\leq
k\leq n}\left\{|\alpha_k| \left\|\sum_{j=1}^{n}
x_j\right\|+\sum_{j=1}^{n}|\alpha_j-\alpha_k| \|x_j\|\right\}$$
if and only if  there are $1\leq i\leq n$ and a state $\varphi$ on $\mathscr{A}$ such that
$${\rm{cis}}\left(\arg\bar{\alpha_i}+\arg(\alpha_k-\alpha_i)\right)\sum_{j=1}^{n}\varphi\langle
x_j,x_k \rangle=\|\sum_{j=1}^{n} x_j\| \|x_k\|$$ for all $k=1,\cdots,n$
satisfying $\alpha_k\neq \alpha_i$. \\
{\rm (ii)} If $\sum_{j=1}^{n} x_j=0$, then\\
$\left\|\sum_{j=1}^{n}\alpha_jx_j\right\|=\min_{1\leq k\leq
n}\left\{\sum_{j=1}^{n}|\alpha_j-\alpha_k| \|x_j\|\right\}$ if and only if\\
there are $i,l\in\{1,\cdots,n\}$ satisfying
$\alpha_i\neq\alpha_l$ and a state $\varphi$ on $\mathscr{A}$ such that
$${\rm{cis}}\left(\arg(\bar{\alpha_l}-\bar{\alpha_i})+\arg(\alpha_k-\alpha_i)\right)\varphi\langle
x_l,x_k \rangle=\|x_l\| \|x_k\|$$ for all $k=1,\cdots,n$
satisfying $k\neq l$ and $\alpha_k\neq \alpha_i$.
\end{corollary}

\begin{proof}
Apply Theorem \ref{T4} by setting $a_j=\alpha_je\,\,
(j=1,\cdots,n)$. Then the results follows from Theorem \ref{T4} and
the following two observations:
$$\frac{\bar{\alpha_i}(\alpha_k-\alpha_i)}{|\alpha_i|
|\alpha_k-\alpha_i|}={\rm{cis}}\left(\arg\bar{\alpha_i}+\arg(\alpha_k-\alpha_i)\right)$$
and
$$\frac{(\bar{\alpha_l}-\bar{\alpha_i})(\alpha_k-\alpha_i)}{|\alpha_l-\alpha_i|
|\alpha_k-\alpha_i|}={\rm{cis}}\left(\arg(\bar{\alpha_l}-\bar{\alpha_i})+\arg(\alpha_k-\alpha_i)\right).$$
\end{proof}

Some special case of Corollary \ref{C1} gives rise to the known results of Pe\v{c}ari\'c and Raji\'c \cite[Corollaries 3.3 and 3.4]{Pec-Raj2}.

\begin{corollary}\label{C2}
Let $\mathscr{X}$ be a pre-Hilbert $\mathscr{A}$-module, $x_1,\cdots,x_n$ be non-zero
elements of $\mathscr{X}$ such that $\|x_i\|\neq\|x_j\|$ for some $i,j.$ \\
{\rm (i)} If $\sum_{j=1}^{n} x_j\neq0$, then
$$\left\|\sum_{j=1}^{n}\frac{x_j}{\|x_j\|}\right\|=\min_{1\leq
k\leq
n}\left\{\frac{1}{\|x_k\|}\left(\left\|\sum_{j=1}^{n}x_j\right\|+\sum_{j=1}^{n}|\|x_j\|-\|x_k\||\right)\right\}$$
if and only if  there are $1\leq i\leq n$ and a state $\varphi$ on
$\mathscr{A}$ such that
$${\rm sgn}(\|x_i\|-\|x_k\|)\sum_{j=1}^{n}\varphi\langle
x_j,x_k \rangle=\left\|\sum_{j=1}^{n} x_j\right\| \|x_k\|$$ for all
$k=1,\cdots,n$
satisfying $\|x_k\|\neq\|x_i\|$.\\
{\rm (ii)} If $\sum_{j=1}^{n} x_j=0$, then\\
$\left\|\sum_{j=1}^{n}\frac{x_j}{\|x_j\|}\right\|=\min_{1\leq k\leq
n}\left\{\frac{1}{\|x_k\|}\sum_{j=1}^{n}|\|x_j\|-\|x_k\||\right\}$ if and only if\\
there are $i,l\in\{1,\cdots,n\}$ satisfying
$\|x_i\|\neq\|x_l\|$ and a state $\varphi$ on $\mathscr{A}$ such that
$${\rm sgn}(\|x_i\|-\|x_l\|){\rm
sgn}(\|x_i\|-\|x_k\|)\varphi\langle x_l,x_k \rangle=\|x_l\|
\|x_k\|$$ for all $k=1,\cdots,n$
satisfying $k\neq l$ and $\|x_k\|\neq\|x_i\|$.
\end{corollary}

\begin{proof}
Apply Corollary \ref{C1} by putting $\alpha_j=\frac{1}{\|x_j\|}\,\,
(j=1,\cdots,n)$. Hence the result follows from Corollary \ref{C1}
and the following two observations:
$${\rm{cis}}\left(\arg\frac{1}{\|x_i\|}+\arg(\frac{1}{\|x_k\|}-\frac{1}{\|x_i\|})\right)={\rm
sgn}(\|x_i\|-\|x_k\|)$$
and
$${\rm{cis}}\left(\arg(\frac{1}{\|x_l\|}-\frac{1}{\|x_i\|})+\arg(\frac{1}{\|x_k\|}-\frac{1}{\|x_i\|})\right)={\rm
sgn}(\|x_i\|-\|x_l\|){\rm sgn}(\|x_i\|-\|x_k\|)\,.$$

\end{proof}

\bibliographystyle{amsplain}

\end{document}